\documentclass{amsart}
\usepackage{appendix}
\usepackage{tabularx}
\usepackage[numbers,sort&compress]{natbib}
\usepackage{color}
\usepackage{amsmath}

\newtheorem{theorem}{Theorem}[section]

\newtheorem{corollary}[theorem]{Corollary}
\newtheorem{definition}[theorem]{Definition}

\newtheorem{conj}[theorem]{Conjecture}
\newtheorem{prop}[theorem]{Proposition}
\newtheorem{question}[theorem]{Question}
\newtheorem{remark}[theorem]{Remark}

\numberwithin{equation}{section}
\allowdisplaybreaks[4]

\begin{document}

\title[Long-time existence of geometric flows]{Long-time existence of some geometric flows with bounded scalar curvature}
%Scalar curvature of some geometric flows
%    Remove any unused author tags.

%    author one information
\author{Chuanhuan Li}
\address{Shanghai Institute for Mathematics and Interdisciplinary Sciences (SIMIS), Shanghai 200433, China\newline
${\quad}$ Research Institute of Intelligent Complex Systems, Fudan University, Shanghai 200433, China}
\curraddr{}
\email{chli@simis.cn}
\thanks{}

%    author two information
\author{Yi Li$^{\ast}$}
\address{Center for Mathematics and Interdisciplinary Sciences, Fudan University, Shanghai 200433, China \newline
${\quad}$ Shanghai Institute for Mathematics and Interdisciplinary Sciences (SIMIS), Shanghai 200433, China}
\curraddr{}
\email{yilicms@simis.cn, yilicms@gmail.com}
\thanks{$^{\ast}$Corresponding author}



\date{}

\dedicatory{dedicated to Professor Shing-Tung Yau's 75th birthday!}

\begin{abstract}
The analysis of singularities in scalar curvature is a classical problem. In this survey, we investigate the behavior of scalar curvature under several geometric flows, with a focus on three specific cases: the Ricci flow, the K\"ahler{-}Ricci flow coupled with $(1,1)$-forms, and the Laplacian flow.
\end{abstract}

\maketitle

%%%%%%%%%%%%%%%%%%%%%%%%%%%%%%%%%%%%%%%%%%%%%%%%
\section{Introduction}
%%%%%%%%%%%%%%%%%%%%%%%%%%%%%%%%%%%%%%%%%%%%%%%%%

Geometric flows play an important role in studying geometric structures, and have been investigated for more than forty years. Let $M$ be a compact smooth $n$-dimensional manifold, the general geometric flow is defined as:
\begin{align}\label{1.1}
    \partial_{t}g(t)=F(g(t)),
\end{align}
where $(g(t))_{t\in[0,T)}$ is a family of Riemannian metrics on $(M, g)$ {with $0<T\leq+\infty$}, $F(g(t))$ is a {symmetric} $2$-tensor. 

${}$
In the study of geometric flows, the theory of existence {of solutions} is a very important research direction. Especially, for the long-time existence of solutions to geometric flow, there is a famous question:

\begin{question}\label{que1.1}
    Let $(M,g(t))_{t\in[0,T)}$ be a solution to general geometric flow $\eqref{1.1}$. If the scalar curvature is uniformly bounded, then {is} $T=+\infty$ ?
\end{question}

Question \ref{que1.1} has seen substantial progress across various geometric flows. In this survey, we focus on three specific cases: Ricci flow on smooth manifolds; K\"ahler-Ricci flow coupled with $(1,1)$-forms on K\"ahler manifolds; Laplacian flow on $7$-dimensional manifolds with closed $G_{2}$-structures.

%%%%%%%%%%%%%%%%%%%%%%%%%%%%%%%%%%%%%%%%%%%%%%%%%%%%%%%%%%%%%%%%%%%%%%%%%%%%
\subsection{Ricci flow}
%%%%%%%%%%%%%%%%%%%%%%%%%%%%%%%%%%%%%%%%%%%%%%%%%%%%%%%%%%%%%%%%%%%%%%%%%%%%
Let $(M,g)$ be a closed, smooth $n$-dimensional Riemannian manifold. A famous result in geometric analysis was established by Hamilton in his seminal paper \cite{Ham82}, where he proved that any three-dimensional closed manifolds admitting a metric of positive Ricci curvature must also admit a metric of constant positive sectional curvature. To achieve this, he introduced the Ricci flow, defined by the evolution equation:
\begin{equation}
  \left \{
       \begin{array}{rl}
          \partial_{t}g(t)&=-2\!\ {\rm Ric}_{g(t)},\\
%           d\varphi(t)&=0,\\
            g(0)&=g,
       \end{array}
  \right.
\end{equation}
where $(g(t))_{t\in[0,T)}$ is a family of Riemannian metrics on $M$. The Ricci flow constitutes a weakly parabolic system of partial differential equations, whose short-time existence and uniqueness {of solutions} on closed manifolds was first proved by Hamilton \cite{Ham82}. {In the noncompact case, short-time existence  was proved by Shi in \cite{Shi89}, and uniqueness was later obtained by Chen and Zhu in \cite{CZ06}. }

A central objective in the study of the Ricci flow is to understand its long-time behavior and the formation of singularities. Key questions include whether the flow exists for all time or develops a finite-time singularity, and in the latter case, characterizing the singularity profile. Such analysis not only reveals the geometric and topological structure of the underlying manifold but also forms the foundation for many major achievements in modern geometric analysis, including the proof of the Poincar\'e and geometrization conjectures. 
%In this survey, we will discuss more details on the long time existence to the solution.
In this survey, we will provide a more detailed discussion of the long-time existence and singularities of solutions.

%\begin{itemize}

%\item {Ricci flow}:

%\end{itemize}

%%%%%%%%%%%%%%%%%%%%%%%%%%%%%%%%%%%%%%%%%%%%%%%%%%%%%%%%%%%%%%%%%%%%%%%%%%%%%%
\subsection{K\"ahler-Ricci flow coupled with $(1,1)$-forms}
%%%%%%%%%%%%%%%%%%%%%%%%%%%%%%%%%%%%%%%%%%%%%%%%%%%%%%%%%%%%%%%%%%%%%%%%%%%%
Let $(\mathcal{X},\omega)$ be a closed K\"ahler manifold with complex dimension $n$. In the study of K\"ahler manifolds, a central problem concerns metrics of constant scalar curvature.
K\"ahler metrics with constant scalar curvature (cscK metric) in the K\"ahler class $[\omega]$ have been studied {for} more {than} forty years. The cscK metric is unique up to {holomorphic} automorphisms \cite{CS14,CT08}. On the other hand, Yau-Tian-Donaldson conjectured that the existence of cscK metrics is equivalent to a certain geometric stability in geometric invariant theory (cf. \cite{Don01,Tian97,Yau93}). 

%${}$

One particularly important class of cscK metrics are K\"ahler-Einstein metrics. When $c_{1}(\mathcal{X})<0$, the existence of K\"ahler-Einstein metrics was independently proven by Aubin \cite{Au78} and Yau \cite{Yau78}. When $c_{1}(\mathcal{X})=0$, the existence of K\"ahler-Einstein metrics was proved by Yau \cite{Yau78}. In this case the K\"ahler-Einstein metrics are Ricci-flat, and are called Calabi-Yau metrics. When $c_{1}(\mathcal{X})>0$, the existence of K\"ahler-Einstein metrics {was shown to be equivalent to some stability condition,} which was independently first proved by Chen-Donaldson-Sun \cite{CDS15A,CDS15B,CDS15C} and Tian \cite{Tian97}. Besides K\"ahler-Einstein metrics, Donaldson proved the existence of the cscK metrics on toric surfaces under some stability conditions \cite{Don09}.

 To understand the existence of constant scalar curvature on K\"ahler manifolds. Yuan, Zhang and the second author \cite{LYZ20} introduced a K\"ahler-Ricci type flow, called K\"ahler-Ricci flow coupled with $(1,1)$-forms:

\begin{equation}
  \left \{
       \begin{array}{rl}
          \partial_{t}\omega(t)&=-\rho(\omega(t))-\omega(t)+\alpha(t), \\
          \partial_{t}\alpha(t)&=\overline{\Box}_{\omega(t)}\alpha(t), \\
           (\omega(0),\alpha(0))&=(\omega,\alpha),
       \end{array}
  \right.\label{flow1.3}
\end{equation}
where $(\omega(t))_{t\in[0,T)}$ is a family of K\"ahler metrics,  $\rho(\omega(t))$ denotes the Ricci form of $\omega(t)$, $\overline{\Box}_{\omega(t)}$   denotes the complex Hodge-Laplace operator, and $(\alpha(t))_{t\in[0,T)}$ is a family of closed $(1,1)$-forms.

{
The motivation to study the flow $\eqref{flow1.3}$ is its connection to constant scalar curvature K\"ahler (cscK) metrics. Suppose that $(\omega_{\infty},\alpha_{\infty})$ is a stationary solution to flow $\eqref{flow1.3}$. Then, in particular, $\alpha_{\infty}$ is a harmonic $(1,1)$-form with respect to $\omega_{\infty}$. This implies that ${\rm tr}_{\omega_{\infty}}(\alpha_{\infty})=$ constant, and therefore, $R_{\omega_{\infty}}=$ constant by equation $\eqref{flow1.3}$.}

%${}$

The authors established that any stationary solution of the flow corresponds to a constant scalar curvature K\"ahler metric, together with a harmonic $(1,1)$-form. The short-time existence of solutions to the flow follows from standard parabolic theory. More precisely, local existence persists even under weaker regularity assumptions specifically, when the underlying manifold  $(\mathcal{X},\omega)$ is merely Hermitian and the form  $\alpha$ is not necessarily closed. For the long-time existence to this flow, we will provide more details in Section \ref{sec3}. For other related works on $\eqref{flow1.3}$, see \cite{FGP2019, ShenSmith2026}.

%%%%%%%%%%%%%%%%%%%%%%%%%%%%%%%%%%%%%%%%%%%%%%%%%%%%%%%%%%%%%%%%%%%%%%%%%%%%
\subsection{Laplacian flow}
%%%%%%%%%%%%%%%%%%%%%%%%%%%%%%%%%%%%%%%%%%%%%%%%%%%%%%%%%%%%%%%%%%%%%%%%%%%%
In $G_2$ geometry, a central objective is the construction of $G_2$-manifolds, {with a particular focus on the irreducible compact case.} Joyce \cite{Joy00} provided the first examples of irreducible compact $G_2$-manifolds by solving a nonlinear elliptic equation. Inspired by Hamilton's work on Ricci flow \cite{Ham82}, Bryant proposed using a geometric flow to construct {torsion-free} $G_2$-structures on $7$-dimensional manifolds. Hence,  he introduced the Laplacian flow, defined as follows:
$$\partial_{ t}\varphi(t)=\Delta_{\varphi(t)}\varphi(t), $$
where $\Delta_{\varphi(t)}\varphi(t)=dd^{\ast}_{\varphi(t)}\varphi(t)+d^{\ast}_{\varphi(t)}d\varphi(t)$ is the Hodge Laplacian of $g(t)$, which is the Riemannian metric algebraically determined by $\varphi(t)$. 
%Since $\Delta_{\varphi}\varphi=dd^{\ast}_{\varphi}\varphi$ for a closed $G_{2}$-structure $\varphi$, we see that the closedness of $\varphi(t)$ is preserved along the Laplacian flow $\eqref{The closed Laplacian flow}$. 

${}$

The flow $\eqref{The closed Laplacian flow}$ can be viewed as the gradient flow for the Hitchin functional when the variations are restricted to the cohomology class of a closed $G_{2}$-structure, where {the} Hitchin functional introduced by Hitchin \cite{Hit00} {is}
$$\mathcal{H}:[\overline{\varphi}]_{+}\longrightarrow\mathbb{R}^{+},\ \varphi\longmapsto\frac{1}{7}\int_{M}\varphi\wedge\psi=\int_{M}\ast_{\varphi}1.$$
Here $\overline{\varphi}$ is a closed $G_{2}$-structure on $M$ and  $[\overline{\varphi}]_{+}$ is the open subset of the cohomology class $[\overline{\varphi}]$ consisting of $G_{2}$-structures. Any critical point of $\mathcal{H}$ gives rise to a torsion-free $G_{2}$-structure when {the} $7$-manifold $M$ is compact, {and they are strict local maxima (transverse to the action of diffeomorphisms).} 
%The stationary points of the Laplacian flow are harmonic forms $\varphi$, which on a compact manifold are the torsion-free $G_{2}$-structure.   
{When the initial $3$-form $\varphi_{0}$ is closed and we evolve inside a fixed cohomology class, Bryant and Xu \cite{BX11} proved the short-time existence and uniqueness of solutions to the following Laplacian flow with closed $G_{2}$-structures. }

{\begin{equation}
  \left \{
       \begin{array}{rl}
          \partial_{ t}\varphi(t)&=\Delta_{\varphi(t)}\varphi(t),\\
           {d\varphi(t)}&{=0,}\\
           \varphi(0)&=\varphi,
       \end{array}
  \right.
\end{equation}}
{where $\varphi$ is an initial  closed $G_{2}$-structure (a positive closed three-form). When restricted to closed forms,  the Laplacian flow also takes the form}
{\begin{equation}
  \left \{
       \begin{array}{rl}
          \partial_{ t}\varphi(t)&=dd^{\ast}\varphi(t),\\
           %{d\varphi(t)}&{=0,}\\
           \varphi(0)&=\varphi,
       \end{array}
  \right.\label{1.5}
\end{equation}}
As for long-time existence and the analysis of singularities of Laplacian flow $\eqref{1.5}$, we will provide a detailed discussion in Section $4$.

%The Laplacian flow can be interpreted as the gradient flow of the Hitchin functional under variations restricted to the cohomology class of the closed $G_2$-structure. On a compact 7-manifold $M$, any critical point of Hitchin functional corresponds to a torsion-free $G_2$-structure. 

%%%%%%%%%%%%%%%%%%%%%%%%%%%%%%%%%%%%%%%%%%%%%%%%%%%%%%%%%%%%%%%%%%%%%%%%%%%%%%
\section{Ricci flow}
%%%%%%%%%%%%%%%%%%%%%%%%%%%%%%%%%%%%%%%%%%%%%%%%%%%%%%%%%%%%%%%%%%%%%%%%%%%

In this section, let $(M,g)$ be a closed Riemannian manifold. we provide a more detailed discussion of long-time existence and the formation of singularities of Ricci flow
\begin{equation}
  \left \{
       \begin{array}{rl}
          \partial_{t}g(t)&=-2\!\ {\rm Ric}_{g(t)},\\
%           d\varphi(t)&=0,\\
            g(0)&=g.
       \end{array}
  \right.
\end{equation}
 Unlike parabolic equations on $\mathbb{R}^n$, long-time existence requires certain curvature assumptions. In \cite{Ham82}, Hamilton established long-time existence under the assumption of bounded Riemannian curvature.

%There are many fruits of Ricci flow. Existence theory plays a fundamental role in the study of partial differential equations. On compact manifolds, the short-time existence and uniqueness of solutions were established by Hamilton in \cite{Ham82}. 

%${}$

\begin{theorem}[\cite{Ham82}]\label{thm2.1}
    Suppose $(M,g(t))_{t\in [0,T)}$ is a solution to Ricci flow. If {$T$ is the maximal time and} $T< +\infty$, then
    $$\limsup_{t\to T}
\left(\max_{M}|{\rm Rm}_{g(t)}|_{g(t)}\right)= +\infty.$$
\end{theorem}
Building on Perelman{'}s noncollapsing theorem for the Ricci flow, \v{S}e\v{s}um \cite{Ses05} improved upon Hamilton's result (Theorem \ref{thm2.1}) and established the following finite-time singularity models.

\begin{theorem}[\cite{Ses05}]\label{thm2.2}
    Suppose $(M,g(t))_{t\in [0,T)}$ is a solution to Ricci flow. If {$T$ is the maximal time and} $T< +\infty$, then
    $$\limsup_{t\to T}
\left(\max_{M}|{\rm Ric}_{g(t)}|_{g(t)}\right)= +\infty.$$
\end{theorem}
Based on local curvature estimates for $|{\rm Rm}_{g(t)}|_{L^{p}(g(t))}$, Kotschwar, Munteanu, and Wang \cite{KMW16} not only recovered the same result in the compact case, but also extended \v{S}e\v{s}um's result (Theorem \ref{thm2.2}) to the noncompact setting.

In the context of integral bounds, Wang \cite{Wang08} and Ye \cite{Ye08} independently established the following {singularity} result.
\begin{theorem}[\cite{Wang08,Ye08}]
    Suppose $(M,g(t))_{t\in [0,T)}$ is a solution to Ricci flow. If {$T$ is the maximal time and}  $T< +\infty$, then
    $$\int^{T}_{0}\int_{M}
|{\rm Rm}_{g(t)}|^{\frac{n+2}{2}}_{g(t)}dV_{g(t)}dt
= +\infty.$$
\end{theorem}

Furthermore, Wang \cite{Wang08} provided an alternative extension for the Ricci flow, demonstrating that
\begin{theorem}[\cite{Wang08}]
     Suppose $(M,g(t))_{t\in [0,T)}$ is a solution to Ricci flow. If {$T$ is the maximal time}, $T< +\infty$ and ${\rm Ric}_{g(t)}\geq-C$ {with a positive constant $C$}, then
     $$\int^{T}_{0}\int_{M}|R_{g(t)}|^{\frac{n+2}{2}}_{g(t)}
dV_{g(t)}dt= + \infty.$$
\end{theorem}

A sufficient condition for extending the Ricci flow beyond time $T$ is the uniform boundedness of either the Riemann or Ricci curvature. Nevertheless, the optimal curvature condition for such an extension remains unknown. This gap in understanding has led to a widely studied conjecture concerning Ricci flow singularities, which asserts the following:
\begin{conj}[\cite{Cao11,Wang08}]\label{conj2.5}
Suppose $(M,g(t))_{t\in [0,T)}$ is a solution to Ricci flow. If {$T$ is the maximal time} and $T< +\infty$, then
$$\limsup_{t\to T}\left(\max_{M}R_{g(t)}\right)
= +\infty.$$
\end{conj}

Up to now, this conjecture has been resolved in several important settings: the three-dimensional case was established independently by Hamilton \cite{Ham82} and Ivey \cite{Ivey93}; the K\"ahler{-}Ricci flow case by Zhang \cite{zhan10}; and the Type-I Ricci flow case by Enders, M\"uller, and Topping \cite{EMT11}.

${}$

By extending Hamilton's technique introduced in \cite{Ham82}, Cao \cite{Cao11} established a curvature pinching estimate for the traceless Ricci curvature tensor and applied it to obtain a partial result  of the conjecture in the general case.
\begin{theorem}[\cite{Cao11}]\label{thm2.7}
Suppose $(M^{n},g(t))_{t\in [0,T)}$ is a solution to Ricci flow with $n\geq 3$, {$T$ is the maximal time}. Let $c$ be a positive constant such that $R_{g(t)}+c>0$ on $M\times[0,T)$. If there exists a positive constant $C$ such that the Weyl tensor satisfies
$$\frac{|W_{g(t)}|_{g(t)}}{R_{g(t)}+c}\leq C$$
and $T< +\infty$, then
$$
\limsup_{t\to T}\left(\max_{M}R_{g(t)}\right)
= +\infty.$$
\end{theorem}

\begin{remark}
By setting
    $$f=\frac{|W_{g(t)}|_{g(t)}^{2}}{(R_{g(t)}+c)^{2}},\quad h=\frac{|{\rm Ric}_{g(t)}|_{g(t)}^{2}}{(R_{g(t)}+c)^{2}}{,}$$
    {given an arbitrary positive number $\epsilon$, we choose a positive constant $C$ such that $\displaystyle \min_{M}R_{g(t)}+C\geq \epsilon$.} Cao's proof for Theorem \ref{thm2.7} implies that  for any $T^{\prime}\in(0,T)$, the inequality
    $$h\leq C_{1}+\frac{1}{\epsilon}\max_{M\times[0,T^{\prime}]}f^{\frac{1}{2}}$$
    holds on $M\times[0,T^{\prime}]$, where $C_{1}$ is a universal constant depending only on $M,g,\epsilon$, and $n$.
\end{remark}

    {Motivated by Cao's result, we make the following assumption:}   for any $T^{\prime}\in(0,T)$, the inequality
    \begin{align}\label{2.2}
        h\geq C_{n,\epsilon}\max_{M\times[0,T^{\prime}]}f^{\frac{1}{2}}-C_{2}
    \end{align}
    holds on $M\times[0,T^{\prime}]$, for some $\epsilon$ and universal constant $C_{2}$ and $C_{n,\epsilon}$ with
    $$
    \frac{1}{\epsilon}>C_{n,\epsilon}>\frac{1}{4}C(n), 
    $$
    $C(n)$ is a universal constant depending only on $n$.

{Based on the above assumption, the second author  proves that Conjecture \ref{conj2.5} is true.}

\begin{theorem}[\cite{Li24}]
Suppose $(M^{n},g(t))_{t\in [0,T)}$ is a solution to Ricci flow with $n\geq 3$, {$T$ is the maximal time}. Let $c$ be a positive constant such that $R_{g(t)}+c>0$ on $M\times[0,T)$. If there exists a positive constant $C(n)$, depending only on $n$, such that
$$\frac{|W_{g(t)}|_{g(t)}}{R_{g(t)}+c}\leq {C_{n,\epsilon}\frac{|{\rm Ric}_{g(t)}|^{2}_{g(t)}}{(R_{g(t)}+c)^{2}}}$$
and $T< +\infty$, then
$$
\limsup_{t\to T}\left(\max_{M}R_{g(t)}\right)
= +\infty.$$
When $n=4$, {we can take $C(4)=1$ and then any constant $C_{4,\epsilon}$ in $(1/4,1/\epsilon)$.}
\end{theorem}
\begin{remark}        
    We observe that $\eqref{2.2}$ can be replaced by the following condition
    $$|{\rm Ric}_{g(t)}|_{g(t)}^{2}\geq C_{n,\epsilon}(R_{g(t)}+c)|W_{g(t)}|_{g(t)}$$
    along the Ricci flow. 
\end{remark}

%%%%%%%%%%%%%

%%%%%%%%%%%%%

%%%%%%%%%%%%%

%%%%%%%%%%%%%

%%%%%%%%%%%%%%%%%%%%%%%%%%%%%%%%%%%%%%%%%%%%%%%%%%%%%%%%%%%%%%%%%%%%%%%%%%%%%%
%{Partial results on Hamilton's conjecture: continued}
%%%%%%%%%%%%%%%%%%%%%%%%%%%%%%%%%%%%%%%%%%%%%%%%%%%%%%%%%%%%%%%%%%%%%%%%%%%%

Additionally, Buzano and Matteo \cite{BD23} also provided a partial result for Conjecture \ref{conj2.5}.
{\begin{theorem}[\cite{BD23}]
Suppose $(M,g(t))_{t\in [0,T)}$ is a solution to Ricci flow on a closed manifold $M$ of dimension $n<8$, $T$ is the maximal time. If $T< +\infty$ and the injectivity radius satisfies
$$
{\rm inj}(M,g(t))\geq \alpha\left(\max_{M\times[0,t]}
|{\rm Ric}_{g(s)}|_{g(s)}\right)^{-1/2}
$$
for a positive constant $\alpha$, then
$$\limsup_{t\to T}\left(\max_{M}R_{g(t)}\right)
= +\infty.$$
%Let $(M,g(t))$ be a Ricci flow on a closed manifold $M$ of dimension $n<8$, defined on $[0,T)$, {$T$ is the maximal time and}  $T<+\infty$. Assume that the scalar curvature is uniformly bounded, and  is bounded from below by
%Suppose $(M^{n},g(t))_{t\in [0,T)}$ is a solution to Ricci flow with $n\geq 3$. If there exists a constant $C$ such that 
%for some $\alpha>0$. Then the flow can be smoothly extended past time $T$.
\end{theorem}}

\begin{remark}
    The restriction for the dimension to manifolds less than $8$ is based on Bamler's work \cite{Bam18}. In fact, under the condition of boundedness of scalar curvature and finite $T$, Bamler \cite{Bam18} proved that there exists an open subset $\Sigma$
 of $M$ such that $g(t)$ converges in $C^{\infty}(\Sigma)$ to a Riemannian metric $g(T)$ on $\Sigma$ as $t\rightarrow T$, and the Hausdorff dimension of $M\backslash \Sigma$, with respect to some pseudo-length metric $d_{T}$ (i.e., the limit of the induced length metric 
$d_{t}$ of $g(t)$) on $M$, is not greater than $n-4$.
\end{remark}

In fact, the evolution equation for scalar curvature under the Ricci flow{:}
$$
\left(\partial_{t}-\triangle_{g(t)}\right)R_{g(t)}
=2|{\rm Ric}_{g(t)}|^{2}_{g(t)}{,}
$$
and the maximum principle imply that $R_{g(t)}$ must have a {lower} bound. Therefore, uniformly bounded scalar curvature corresponds to the condition $R_{g(t)}\leq C$ for some positive constant $C$.
In the case where $(M,g)$ is a closed $4$-dimensional Riemannian manifold, Simon \cite{Sim20} and Bamler{-}Zhang \cite{BZ17} independently established the following result:
\begin{theorem}[\cite{BZ17,Sim20}]
Suppose $(M^{4},g(t))_{t\in [0,T)}$ is a solution to Ricci flow with the maximal time $T$. If there exists a constant $c$ such that $R_{g(t)}\leq c$ and $T< +\infty$, then there exists a constant $C$, such that
    $$
\int_{M}|{\rm Rm}_{g(t)}|^{2}_{g(t)}dV_{g(t)}
\leq C.
$$
\end{theorem}

%%%%%%%%%%%%%%%%%%%%%%%%%%%%%%%%%%%%%%%%%%%%%%%%%%%%%%%%%%%%%%%%%%%%%%%%%%%%%%
%%%%%%%%%%%%%%%%%%%%%%%%%%%%%%%%%%%%%%%%%%%%%%%%%%%%%%%%%%%%%%%%%%%%%%%%%%%%%%
\section{A flow on cscK metrics in K\"ahler geometry}\label{sec3}
%%%%%%%%%%%%%%%%%%%%%%%%%%%%%%%%%%%%%%%%%%%%%%%%%%%%%%%%%%%%%%%%%%%%%%%%%%%%%%
%%%%%%%%%%%%%%%%%%%%%%%%%%%%%%%%%%%%%%%%%%%%%%%%%%%%%%%%%%%%%%%%%%%%%%%%%%%%%%

{Let $(\mathcal{X}, \omega)$ be a closed K\"ahler manifold of complex dimension $n$.} 
% To investigate the existence problem of constant scalar curvature K\"ahler (cscK) metrics, Li, Yuan, and Zhang \cite{LYZ20} introduced a new parabolic flow:
In this section, we continue study the K\"ahler-Ricci flow coupled with $(1,1)$-forms:

\begin{equation}\label{2}
  \left \{
       \begin{array}{rl}
          \partial_{t}\omega(t)&=-\rho(\omega(t))-\omega(t)+\alpha(t), \\
          \partial_{t}\alpha(t)&=\overline{\Box}_{\omega(t)}\alpha(t), \\
           (\omega(0),\alpha(0))&=(\omega,\alpha),
       \end{array}
  \right.
\end{equation}
%where $\alpha$ is {a} closed real $(1,1)$-form on $\mathcal{X}$, $\rho(\cdot)$ is the Ricci form of $\omega(t)$, $\overline{\Box}_{\omega(t)}$ is the complex Hodge-Laplace operator.
%This flow is a modified K\"ahler-Ricci flow coupled with a heat flow for $(1,1)$-forms, whose stationary solution is a cscK metric, coupled with a harmonic $(1,1)$-form. 

Taking the cohomology class of (\ref{2}), we have
\begin{equation*}
\frac{d}{dt}[\omega(t)]=-2\pi c_{1}(\mathcal{X})
-[\omega(t)]+[\alpha], \ \ \ \frac{d}{dt}[\alpha(t)]=0,
\end{equation*}
which implies
\begin{equation*}
[\omega(t)]+2\pi c_{1}(\mathcal{X})-[\alpha]
=e^{-t}\left([\omega]+2\pi c_{1}(\mathcal{X})-[\alpha]\right), \ \ \
[\alpha(t)]=[\alpha].
\end{equation*}
To get a scalar calculus, we should impose a topological condition:
{\begin{equation}
\omega\in-2\pi c_{1}(\mathcal{X})+[\alpha].\label{5}
\end{equation}}
{Suppose that $-2\pi c_{1}(\mathcal{X})+[\alpha]\in\mathcal{K}_{\mathcal{X}}$, where $\mathcal{K}_{\mathcal{X}}$ is the K\"ahler cone of $\mathcal{X}$. Assume that $\omega\in-2\pi c_{1}(\mathcal{X})+[\alpha]$, one can show that the solution $\omega(t)$ also
lies in the class $-2\pi c_{1}(\mathcal{X})+[\alpha]$.}

Under (\ref{5}), one has
\begin{equation*}
\alpha(t)=\alpha+\sqrt{-1}
\partial\bar{\partial}f(t), \ \omega(t)=\omega+\sqrt{-1}\partial\bar{\partial}
\varphi(t), \ \omega=\alpha+\sqrt{-1}\partial\bar{\partial}
\ln\Omega
\end{equation*}
where $f(t), \varphi(t)\in C^{\infty}(\mathcal{X},\mathbb{R})$ and $\Omega$
is a smooth volume form. Then flow (\ref{2}) is equivalent to the parabolic {complex Monge-Amp\`ere} equation coupled with a heat equation:

\begin{equation}
  \left \{
       \begin{array}{rl}
          \partial_{t}\varphi(t)&=\displaystyle{\ln\frac{(\omega+\sqrt{-1}
\partial\bar{\partial}\varphi(t))^{n}}{\Omega}
-\varphi(t)+f(t)},\nonumber \\
          \partial_{t}f(t)&=\Delta_{\omega(t)}f(t)
+{\rm tr}_{\omega(t)}\alpha \ \ = \ \ {\rm tr}_{\omega(t)}\alpha(t),
\\
          (\varphi(0),f(0))&=(0,0){.}
       \end{array}
  \right.
\end{equation}

Note that if $c_{1}(\mathcal{X})<0$ and $\alpha=0$, then the flow (\ref{2}) reduces to the normalized K\"ahler-Ricci flow on $\mathcal{X}$, which converges exponentially to the negative K\"ahler-Einstein metric coupled with $\alpha_{\infty}=0$. For long-time existence result, Li,Yuan and Zhang \cite{LYZ20} established the following result using Shi-type estimates.

\begin{theorem}[\cite{LYZ20}]\label{thm3.1}
    Suppose that $(\mathcal{X},\omega)$ is a closed K\"ahler manifold and $\alpha$ is a closed nonnegative $(1,1)$-form such that
    \begin{align}\label{3.3}
        \omega\in-2\pi c_{1}(\mathcal{X})+[\alpha]{.}
    \end{align}
    Let $(\omega(t),\alpha(t))$ be the solution to the flow (\ref{2}) on the maximal time interval $[0,T)$ for $T< +\infty$ with the initial condition $(\omega,\alpha)$. Then
    $$\limsup_{t\to T}\left(\max_{\mathcal{X}}|{\rm Rm}_{\omega(t)}|_{\omega(t)}\right)=+\infty{.}$$
\end{theorem}
 By employing a blow-up argument and the Cheeger{-}Gromov compactness theorem, they also established {the following.}
 \begin{theorem}[\cite{LYZ20}]
      Assume that $\alpha$ is a closed nonnegative $(1,1)$-form such that
    $$\omega\in-2\pi c_{1}(\mathcal{X})+[\alpha]{.}$$
    Let $(\omega(t),\alpha(t))$ be the solution to the flow (\ref{2}) on $[0,T)$ for $T< +\infty$ with the initial condition $(\omega,\alpha)$. Suppose that the Ricci curvature of $\omega(t)$ and $|\alpha(t)|_{\omega(t)}$ are uniformly bounded on $[0,T)$. Then the solution $(\omega(t),\alpha(t))$ can be extended past
    time $T$.
 \end{theorem}

In analogy with the Ricci flow, a natural goal is to determine the optimal condition under which flow (\ref{2}) can be extended. Toward this end, the authors proposed the following conjecture:
\begin{conj}
Under the same topological condition in Theorem \ref{thm3.1}, the flow (\ref{2}) exists for all time as along
as $R_{\omega(t)}$ and ${\rm tr}_{\omega(t)}\alpha(t)$ are
uniformly bounded.
\end{conj}

\begin{remark}
    For the Calabi flow, Chen and Cheng \cite{CC21A, CC21B, CC18} made a breakthrough on this conjecture using the elliptic approach; however, the parabolic approach remains open.
\end{remark}

Taking the cohomology classes on both sides of (\ref{2}), we have
%$$
%\frac{d}{dt}[\omega(t)]=-2\pi c_{1}(\mathcal{X})-[\omega(t)]
%+[\alpha(t)], \ \ \ [\alpha(t)]=[\alpha].
%$$
%Hence
%From $\eqref{3.2}$, we have
$$
[\omega(t)]=(1-e^{-t})\left([\alpha]-[\omega]-2\pi c_{1}(
\mathcal{X})
\right)+[\omega].
$$
Motivated by K\"ahler-Ricci flow, the authors gave the following conjecture:
\begin{conj}\label{conj3.5}
Let $(\omega(t),\alpha(t))$ be the solution to the flow (\ref{2}) on the maximal time interval $[0,T)$. The maximal time $T$ can be characterized as
$$
T=\sup\left\{t\in(0,+\infty]:
(1-e^{-t})([\alpha]-[\omega]-2\pi c_{1}(\mathcal{X}))+[\omega]>0\right\}.
$$
In particular,
\begin{itemize}
    \item[(i)] If ${[\omega]\leq[\alpha]-2\pi c_{1}(\mathcal{X})}$, then time $T$ satisfies $T=+\infty$.
    \item[(ii)] If ${[\omega]>[\alpha]-2\pi c_{1}(\mathcal{X})} $ \ and \ $[\alpha]\geq 2\pi c_{1}(\mathcal{X})$, then time $T$ satisfies $T=+\infty$.
     \item[(iii)] If ${[\omega]>[\alpha]-2\pi c_{1}(\mathcal{X})} $ \ and \ $[\alpha]< 2\pi c_{1}(\mathcal{X})$, then time $T$ satisfies
     $$T
=\ln\left(1+\frac{[\omega]}{2\pi c_{1}(\mathcal{X})-[\alpha]}\right).$$
\end{itemize}
\end{conj}
\begin{remark}
    When $\alpha(t)\equiv0$, flow (\ref{2}) will reduce to the K\"ahler-Ricci flow. For this case, Conjecture \ref{conj3.5} has been proved by  Cao \cite{Cao85} for the first case $(i)$, Tsuji \cite{Ts88} for the second case $(ii)$, Tian and Zhang \cite{TZ06} for the third case $(iii)$.
\end{remark}

Using the local curvature estimate, the second author and Yuan generalized the long time existence of flow (\ref{2}) {without the condition $\eqref{3.3}$.}
\begin{theorem}[\cite{LY21}]
Let $(\omega(t),\alpha(t))$ be the solution to the flow (\ref{2}) on the maximal time interval $[0,T)$. %for $T<\infty$ 
%under the condition 
%\begin{align}\label{3.3}
%    -2\pi c_{1}(\mathcal{X})+[\alpha]-[\omega]=0.
%\end{align}
If $T<+\infty$, then
$$
\limsup_{t\to T}\left(\max_{\mathcal{X}}\left\{|\rho(\omega(t))|_{\omega(t)},
|\alpha(t)|_{\omega(t)}\right\}\right)=+\infty.
$$
\end{theorem}

As a consequence, they {obtain the following} 

\begin{theorem}[\cite{LY21}]
Let $(\omega(t),\alpha(t))$ be the solution to the flow (\ref{2}) for $t \in [0, T)$ with the initial condition $(\omega,\alpha)$. 
%Under the condition $\eqref{3.3}$. 
Suppose that the Ricci curvature of $\omega(t)$ and $|\alpha(t)|_{\omega(t)}$ are uniformly bounded on $[0, T)$. If 
 $T<+\infty$, then the solution $(\omega(t),\alpha(t))$ can be extended past time $T$.
\end{theorem}

%%%%%%%%%%%%%%%%%%%%%%%%%%%%%%%%%%%%%%%%%%%%%%%%%%%%%%%%%%%%%%%%%%%%%%%%%%%%%%
\section{Laplacian flow}
%%%%%%%%%%%%%%%%%%%%%%%%%%%%%%%%%%%%%%%%%%%%%%%%%%%%%%%%%%%%%%%%%%%%%%%%%%%%

%%%%%%%%%%%%%%%%%%%%%%%%%%%%%%%%%%%%%%%%%%%%%%%%%%%%%%%%%%%%%%%%%%%%%%%%%%%%%%
%%%%%%%%%%%%%%%%%%%%%%%%%%%%%%%%%%%%%%%%%%%%%%%%%%%%%%%%%%%%%%%%%%%%%%%%%%%%%%
\subsection{$G_{2}$-structures on $7$-manifolds}
%%%%%%%%%%%%%%%%%%%%%%%%%%%%%%%%%%%%%%%%%%%%%%%%%%%%%%%%%%%%%%%%%%%%%%%%%%%
Let $\mathbb{O}$ be the octonions (exceptional division algebra). {From} the vector cross product ``$\times$" on ${\rm Im}\ \mathbb{O}$, we can define the 3-form by
$$\phi(a,b,c):=\frac{1}{2}\langle {[}a,b{]},c\rangle=\langle a\times b,c\rangle\quad\quad\quad\text{for}\ a,b,c\in {\rm Im}\ \mathbb{O}.$$
Let $\{e_{1},e_{2},\cdots,e_{7}\}$ denote the standard basis of $\mathbb{R}^{7}$ and $\{e^{1},e^{2},\cdots,e^{7}\}$ be its dual basis. Using the octonion multiplication table, one can show that
\begin{align}\label{phi}
    \phi=e^{123}+e^{145}+e^{167}+e^{246}-e^{257}-e^{347}-e^{356},
\end{align}
where $e^{ijk}:=e^{i}\wedge e^{j}\wedge e^{k}$. The subgroup {fixing $\phi$} of ${\rm GL}(7,\mathbb{R})$ is the exceptional Lie group $G_{2}$, which is a compact, connected, simple $14$-dimensional Lie subgroup of ${\rm SO}(7)$. In fact, {Lie group} $G_{2}$ acts irreducibly on $\mathbb{R}^{7}$ and preserves the metric and orientation for which $\{e_{1},e_{2},\cdots,e_{7}\}$ is an oriented orthonormal basis. Note that {Lie group} $G_{2}$ also preserves the $4$-form
$$\ast_{\phi}\phi=e^{4567}+e^{2367}+e^{2345}+e^{1357}-e^{1346}-e^{1256}-e^{1247},$$
where $\ast_{\phi}$ is the Hodge star operator determined by the metric and orientation.

\begin{remark}
    The vector cross product ``$\times$" is an algebraic structure defined in {the imaginary part of} a normed division algebra. Therefore, the $G_{2}$-structure can only be defined in the 7-dimensional manifold. For more details, see \cite{Kar20}.
\end{remark}

For a smooth $7$-manifold $M$ and a point $x\in M$, we define 
\begin{align}
    \wedge_{+}^{3}(T_{x}^{\ast}M):=\left\{\varphi_{x}\in\wedge^{3}(T_{x}^{\ast}M)\ \Big{|}\ h^{\ast}\phi=\varphi_{x},\
    \text{for\ invertible}\ h\in\text{Hom}_{\mathbb{R}}(T_{x}^{\ast}M,\mathbb{R}^{7})\right\}\notag
\end{align}
and the bundle
\begin{align}
    \wedge_{+}^{3}(T^{\ast}M):=\bigcup_{x\in M}\wedge_{+}^{3}(T_{x}^{\ast}M)\notag.
\end{align}
We call a section $\varphi$ of $\wedge_{+}^{3}(T^{\ast}M)$ a {\it positive $3$-form} on $M$ or a {\it $G_{2}$-structure} on $M$, and denote the space of positive 3-form{s} by $\Omega^{3}_{+}(M)$. The existence of $G_{2}$-structure is equivalent to the
property that $M$ is orientable and spinnable, which is equivalent to the vanishing of the first two Stiefel-Whitney classes $w_{1}(TM)$ and $w_{2}(TM)$. For more details, see Theorem 10.6 in \cite{LW89}.

For a $3$-form $\varphi$, we define a $\Omega^{7}(M)$-valued bilinear form $\text{B}_{\varphi}$ by
$$\text{B}_{\varphi}(u,v)=\frac{1}{6}(u\lrcorner\varphi)\wedge(v\lrcorner\varphi)\wedge\varphi,$$
where $u, v$ are tangent vectors on $M$ and $``\lrcorner"$ is the interior multiplication operator (\textit{Here we use the orientation in} \cite{Bry05}, {which refers to $\eqref{phi}$}). Then we can see that any $\varphi\in\Omega^{3}_{+}(M)$ determines a Riemannian metric $g_{\varphi}$ and an orientation $d V_{\varphi}$, hence the Hodge star operator $\ast_{\varphi}$ and the associated $4$-form
$$\psi:=\ast_{\varphi}\varphi$$
can also be uniquely determined by $\varphi$. 

The group $G_{2}$ acts irreducibly on $\mathbb{R}^{7}$ (and hence on $\wedge^{1}(\mathbb{R}^{7})^{\ast}$ and $\wedge^{6}(\mathbb{R}^{7})^{\ast}$), but it
acts reducibly on $\wedge^{k}(\mathbb{R}^{7})^{\ast}$ for $2\leq k\leq 5$. Hence a $G_{2}$ structure $\varphi$ induces splittings
of the bundles $\wedge^{k}(T^{\ast}M)(2\leq k\leq5)$ into direct summands, which we denote by
$\wedge^{k}_{l}(T^{\ast}M,\varphi)$ with $l$ being the rank of the bundle. We let the space of sections
of $\wedge^{k}_{l}(T^{\ast}M,\varphi)$ be $\Omega^{k}_{l}(M)$. Define the natural projections
$$\pi^{k}_{l}:\Omega^{k}(M)\longrightarrow \Omega^{k}_{l}(M),\ \ \alpha\longmapsto \pi^{k}_{l}(\alpha).$$
Then we have
\begin{align}
    \Omega^{2}(M)&=\Omega^{2}_{7}(M)\oplus\Omega^{2}_{14}(M),\notag\\
    \Omega^{3}(M)&=\Omega^{3}_{1}(M)\oplus\Omega^{3}_{7}(M)\oplus\Omega^{3}_{27}(M)\notag{,}
\end{align}
where each component is determined by
\begin{align}
    \Omega^{2}_{7}(M)&=\{X\lrcorner\varphi:X\in C^{\infty}(TM)\}=\{\beta\in\Omega^{2}(M):\ast_{\varphi}(\varphi\wedge\beta)=2\beta\},\notag\\
    \Omega^{2}_{14}(M)&=\{\beta\in\Omega^{2}(M):\psi\wedge\beta=0\}=\{\beta\in\Omega^{2}(M):\ast_{\varphi}(\varphi\wedge\beta)=-\beta\},\notag
\end{align}
and
\begin{align}
    \Omega^{3}_{1}(M)&=\{f\varphi:f\in C^{\infty}(M)\},\notag\\
    \Omega^{3}_{7}(M)&=\{\ast_{\varphi}(\varphi\wedge\alpha):\alpha\in\Omega^{1}(M)\}=\{X\lrcorner\psi:X\in C^{\infty}(TM)\},\notag\\
    \Omega^{3}_{27}(M)&=\{\eta\in\Omega^{3}(M):\eta\wedge\varphi=\eta\wedge\psi=0\}.\notag
\end{align}

\begin{remark}\label{remark2.1}
    $\Omega^{4}$ and $\Omega^{5}$ have the corresponding decompositions by Hodge duality. {For} more details about $G_{2}$-decompositions, see \cite{Bry05, Kar20}.
\end{remark}

By the definition  of $G_{2}$ decompositions, we can find unique differential forms
$\tau_{0}\in\Omega^{0}(M),\tau_{1},\widetilde{\tau}_{1}\in\Omega^{1}(M),\tau_{2}\in\Omega^{2}_{14}(M)$ and $\tau_{3}\in\Omega^{3}_{27}(M)$ such that (see \cite{Bry05})
\begin{align}
    d\varphi&=\tau_{0}\psi+3\!\ \tau_{1}\wedge\varphi+\ast_{\varphi}\tau_{3},\\
    d\psi&=4\!\ \widetilde{\tau}_{1}\wedge\psi+\tau_{2}\wedge\varphi.
\end{align}
In fact, Bryant \cite{Bry05} proved that $\tau_{1}=\widetilde{\tau}_{1}$. We call $\tau_{0}$ the \textit{scalar torsion}, $\tau_{1}$ the \textit{vector torsion}, $\tau_{2}$ the \textit{Lie algebra torsion}, and $\tau_{3}$ the \textit{symmetric traceless torsion}. We also call $\tau_{\varphi}:=\{\tau_{0},\tau_{1},\tau_{2},\tau_{3}\}$ the intrinsic torsion forms of the $G_{2}$-structure $\varphi$.

{From \cite{Kar20}, the covariant derivative $\nabla\varphi$ is a smooth section of $T^{\ast}M\otimes \Lambda^{3}_{7}(T^{\ast}M)$. Thus we can define the full torsion tensor as 
\begin{definition}
    Let $X$ be a vector field on $M$. $\nabla_{X}\varphi$ can be written as
    $$\nabla_{X}\varphi=\mathbf{T}(X)\lrcorner \psi$$
    for some vector field $\mathbf{T}(X)$ on $M$. We call $\mathbf{T}=\mathbf{T}_{ij}dx^{i}\otimes dx^{j}$ the full torsion tensor of $\varphi$, which satisfies
    \begin{align}\label{2.7}
    \nabla_{i}\varphi_{jkl}&=\mathbf{T}_{i}^{\ m}\psi_{mjkl},\\
    \mathbf{T}_{i}^{\ j}&=\frac{1}{24}\nabla_{i}\varphi_{lmn}\psi^{jlmn},
\end{align}
and 
\begin{align}
    \nabla_{m}\psi_{ijkl}=-(\mathbf{T}_{mi}\varphi_{jkl}-\mathbf{T}_{mj}\varphi_{ikl}-\mathbf{T}_{mk}\varphi_{jil}-\mathbf{T}_{ml}\varphi_{jki}).
\end{align}
\end{definition}
%the full torsion tensor $\mathbf{T}=\mathbf{T}_{ij}dx^{i}\otimes dx^{j}$ satisfies the followings:
The full torsion tensor $\mathbf{T}_{ij}$ is related to the
intrinsic torsion forms by the following:
\begin{align}
\label{Def of T}\mathbf{T}_{ij}=\frac{\tau_{0}}{4}g_{ij}-(\tau_{3})_{ij}-(\tau_{1}^{\#}\lrcorner\varphi)_{ij}-\frac{1}{2}(\tau_{2})_{ij}
\end{align}
or as $2$-tensors,
$$\mathbf{T}=\frac{\tau_{0}}{4}g_{\varphi}-\tau_{3}-\tau_{1}^{\#}\lrcorner\varphi-\frac{1}{2}\tau_{2},$$
where $(\tau_{1}^{\#}\lrcorner\varphi)_{ij}=(\tau_{1}^{\#})^{l}\varphi_{lij}$ and $\#$ is the isomorphism from $1$-form to vector fields.}

${}$

If $\varphi$ is closed, which means $d\varphi=0$, then $\tau_{0},\tau_{1},\tau_{3}$ are all zero, so the only nonzero torsion form is 
$$\tau\equiv\tau_{2}=(\tau_{2})_{ij}dx^{i}\otimes dx^{j}=\frac{1}{2}\tau_{ij}dx^{i}{\wedge} dx^{j}.$$
 Then from {\cite{LW17}}, the full torsion tensor $\mathbf{T}$ satisfies the following.
$$
\mathbf{T}_{ij}=-\mathbf{T}_{ji}=-\frac{1}{2}(\tau_{2})_{ij}\ \ \text{or equivalently}\ \ \mathbf{T}=-\frac{1}{2}\tau,
$$
so that $\mathbf{T}$ is a skew-symmetric 2-tensor or a $2$-form.

%%%%%%%%%%%%%%%%%%%%%%%%%%%%%%%%%%%%%%%%%%%%%%%%%%%%%%%%%%%%%%%%%%%%%%%%%%%%%%
\subsection{Laplacian flow}
%%%%%%%%%%%%%%%%%%%%%%%%%%%%%%%%%%%%%%%%%%%%%%%%%%%%%%%%%%%%%%%%%%%%%%%%%%%

In this section, we study   the following Laplacian flow for closed $ G_{2}$-structure{s} introduced by Bryant \cite{Bry05}.
%\begin{equation}
%  \left \{
%       \begin{array}{rl}
%          \partial_{ t}\varphi(t)&=\Delta_{\varphi(t)}\varphi(t),\\
%           d\varphi(t)&=0,\\
 %          \varphi(0)&=\varphi,
%       \end{array}
%  \right.
%  \label{The closed Laplacian flow}
%\end{equation}
{\begin{equation}
  \left \{
       \begin{array}{rl}
          \partial_{ t}\varphi(t)&=dd^{\ast}\varphi(t),\\
           %{d\varphi(t)}&{=0,}\\
           \varphi(0)&=\varphi,
       \end{array}
  \right.
  \label{The closed Laplacian flow}
\end{equation}}
where %$\Delta_{\varphi(t)}\varphi(t)=dd^{\ast}_{\varphi(t)}\varphi(t)+d^{\ast}_{\varphi(t)}d\varphi(t)$ is the Hodge Laplacian of $g(t)$ and 
$\varphi$ is an initial  closed $G_{2}$-structure(a positive closed three-form). 
%Here $g(t)$ is the Riemannian metric {algebraically determined by} $\varphi(t)$. Since $\Delta_{\varphi}\varphi=dd^{\ast}_{\varphi}\varphi$ for a closed $G_{2}$-structure $\varphi$, we see that the closedness of $\varphi(t)$ is preserved along the Laplacian flow $\eqref{The closed Laplacian flow}$. 
%The flow $\eqref{The closed Laplacian flow}$ can be viewed as the gradient flow for the Hitchin functional when the variations are restricted to the cohomology class of the closed $G_{2}$-structure, where {the} Hitchin functional introduced by Hitchin \cite{Hit00} {is}
%$$\mathcal{H}:[\overline{\varphi}]_{+}\longrightarrow\mathbb{R}^{+},\ \varphi\longmapsto\frac{1}{7}\int_{M}\varphi\wedge\psi=\int_{M}\ast_{\varphi}1.$$
%Here $\overline{\varphi}$ is a closed $G_{2}$-structure on $M$ and  $[\overline{\varphi}]_{+}$ is the open subset of the cohomology class $[\overline{\varphi}]$ consisting of $G_{2}$-structures. Any critical point of $\mathcal{H}$ gives rise to a torsion-free $G_{2}$-structure when {the} $7$-manifold $M$ is compact.
The motivation for Bryant {to} pose {the Laplacian flow} is from a reasonable {expectation}, which is based on the work of Joyce \cite{Joy00}. {In this survey, we call it reasonable conjecture.}

\begin{conj} If the initial $G_{2}$-structure $\varphi$ on a closed $7$-manifold is closed and has {sufficiently small torsion}, then the flow \eqref{The closed Laplacian flow} will exist for all time and converge to a torsion-free $G_{2}$-structure.
\end{conj}

{Motivated by ``Ricci-flat metrics are dynamically stable along the Ricci flow", Lotay and Wei \cite{LW19A} gave a positive answer for the above conjecture.}

%%%%%%%%%%%%%%%%%%%%%%%%%%%%%%%%%%%%%%%%%%%%%%%%%%%%%%%%%%%%%%%%%%%%%%%%%%%%%%
\subsection{Existence results for the solution}
%%%%%%%%%%%%%%%%%%%%%%%%%%%%%%%%%%%%%%%%%%%%%%%%%%%%%%%%%%%%%%%%%%%%%%%%%%%
When studying a geometric flow, a fundamental question is the existence of its solutions. In 2011, Bryant and Xu \cite{BX11} established the short-time existence and uniqueness of solutions to the Laplacian flow.

\begin{theorem}[\cite{BX11}]
For a {closed} $7$-manifold $M$, the initial value problem \eqref{The closed Laplacian flow} has a unique solution for a short time $[0,T)$ with $T$ depending on $\varphi$.
\end{theorem}

As for the long-time solution, Lotay and Wei \cite{LW17} established a key result under the assumption that either the Riemannian curvature or the Ricci curvature remains bounded.

\begin{theorem} [\cite{LW17}] \label{thm4.5}
Let $M$ be a {closed} $7$-manifold and $(\varphi_{t})_{t\in[0,T)}$ be a solution to
the flow \eqref{The closed Laplacian flow} for closed $G_{2}$-structures {with $T< +\infty$.}

\begin{itemize}

\item[(a)] If the Riemann curvature satisfies 
$$\sup_{M\times[0,T)}\left(|{\rm Rm}_{\varphi(t)}|_{\varphi(t)}
+|\nabla_{\varphi(t)}{\bf T}(t)|_{\varphi(t)}\right)< +\infty.$$ then $\varphi(t)$ can be extended past time $T$.

\item[(b)] If the Ricci curvature satisfies 
$$\sup_{M\times[0,T)}|{\rm Ric}_{\varphi(t)}|_{\varphi(t)}< +\infty,$$ then $\varphi(t)$ can be extended past time $T$.
\end{itemize}
\end{theorem}

{
\begin{remark}\label{rem4.7}
    From Proposition 2.4 in \cite{LW17} or $(2.33)$ in \cite{Li21}, when $G_{2}$-structure is closed, $\nabla\mathbf{T}$ can be controlled by Riemann curvature ${\rm Rm}$ and torsion $|\mathbf{T}|$. Note that $|\mathbf{T}|^{2}=-R$, the norm of torsion can be controlled by Riemann curvature. Hence the assumption of $(a)$ can be replaced by uniformly bounded Riemann curvature.
\end{remark}}

In a series of papers \cite{LW17,LW19A,LW19B}, they also established many foundational results, including the stability of torsion-free $G_2$-structures, the real-analyticity of closed $G_2$-structures under the Laplacian flow, among others. Based on the following curvature estimates, the second author \cite{Li21} later provided a new and simplified proof of part $(b)$ in Theorem \ref{thm4.5}.

\begin{theorem}[\cite{Li21}]
Let $M$ be a smooth $7$-manifold and $(\varphi(t))_{t\in[0,T)}$ be a solution to the flow \eqref{The closed Laplacian flow} for closed $G_{2}$-structures {with $T< +\infty$}. Assume there exist constants $A, K>0$ and a point $x_{0}\in M$ such that the geodesic ball $B_{g}(x_{0}, A/\sqrt{K})$ is compactly contained in $M$ { with $g=g_{\varphi(0)}$} and that
\begin{equation*}
|{\rm Ric}_{\varphi(t)}|_{\varphi(t)}\leq K \ \ \ \text{on} \
{B_{g}(x_{0},A/\sqrt{K})}
\times[0,T].
\end{equation*}
Then, for any {$p\geq5$}, there exists $c=c(p)>0$ so that, for
all $t\in[0,T]$,
$$
\int_{{B_{g}(x_{0},A/2\sqrt{K})}}
|{\rm Rm}_{\varphi(t)}|^{p}_{\varphi(t)}dV_{\varphi(t)} \ \leq \ c(1+K)
e^{cKT}\int_{{B_{g}(x_{0},A/\sqrt{K})}}
|{\rm Rm}_{\varphi}|^{p}_{\varphi}dV_{\varphi}
$$
$$
+ \ c K^{p}\left(1+A^{-2p}\right) e^{cKT}
{\rm Vol}_{\varphi(t)}({B_{g}
(x_{0},A/\sqrt{K})}).
$$
\end{theorem}

Moreover, Chen \cite{Chen18} investigated the finite-time singularities of the Laplacian flow and established the following blow-up rate for curvatures.

\begin{theorem}[\cite{Chen18}]\label{thm4.9}
Let $M$ be a closed $7$-manifold and $(\varphi(t))_{t\in[0,T)}$ be a solution to the
flow \eqref{The closed Laplacian flow} for closed $G_{2}$-structures {on the maximal time interval $[0,T)$}. If  $T< +\infty$, then
$$
\max_{M}\left(|{\rm Rm}_{\varphi(t)}|^{2}_{\varphi(t)}+|{\bf T}(t)|^{4}_{\varphi(t)}+|\nabla_{\varphi(t)}{\bf T}(t)|^{2}_{\varphi(t)}\right)^{1/2}\geq\frac{C}{T-t}
$$
for some constant $C>0$. Consequently,
$$
\int^{T}_{0}\max_{M}\left(|{\rm Ric}_{\varphi(t)}|_{\varphi(t)}+|{\bf T}(t)|^{2}_{\varphi(t)}\right)dt=+\infty.
$$
\end{theorem}

When the $G_{2}$-structure is closed, from Remark \ref{rem4.7}, some assumptions of Theorem \ref{thm4.9} can be replaced by Riemann curvature or Ricci curvature.
%the Riemann curvature and Ricci curvature are related to the first covariant derivative of the torsion tensor, while the scalar curvature is related to the norm of the torsion tensor. 
To find the optimal condition under which Laplacian flow can be extended. Lotay and Wei \cite{LW17} propose the following {Question}:

\begin{question}[\cite{LW17}]\label{conj4.8}
Let $M$ be a closed $7$-manifold and $(\varphi(t))_{t\in[0,T)}$ be a solution to the
flow \eqref{The closed Laplacian flow} for closed $G_{2}$-structures {on the maximal time interval $[0,T)$}.
    If $T< +\infty$, then 
    $$\liminf_{t\to T}(\min_{M}R_{\varphi(t)})=-\infty.$$
\end{question}

{There are some positive answers about this question on some special manifolds.}
Fine and Yao \cite{FY18} proved the case where $M^{7}=\mathbb{T}^{3}\times X^{4}$ where $X^{4}$ is equipped with a hypersymplectic structure, $\mathbb{T}^{3}$ is the standard $3$-dimensional torus.
For the Laplacian flow on $\mathbb{ T}^{3}\times\mathcal{X}^{4}$, or on $\mathbb{S}^{1}\times\mathcal{X}^{6}$ (where $\mathcal{X}^{2n}$ is a Calabi-Yau $n$-fold), Picard and Suan \cite{PS24} established a relation between the Laplacian flow and complex Monge-Amp\`ere flow
    $$   \partial_{t}u(t)=3\left[e^{-2\ln|\Omega|_{\omega}}\frac{\det(\omega+\sqrt{-1}\partial\overline{\partial}u(t))}{
    \det\omega}\right]^{1/3}, \ \ \ u(0)=0,
    $$
    where $(\omega,\Omega)$ is the (K\"ahler metric, holomorphic volume form) on $\mathcal{X}^{4}$ or $\mathcal{X}^{6}$. {In this case, they proved the convergence of the Laplacian flow.}

The evolution equation for $g(t)$ is given by
$$
\partial_{t}g(t)=-2{\rm Ric}_{\varphi(t)}-\frac{4}{3}|{\bf T}(t)|^{2}_{\varphi(t)}g(t)
-4{\bf T}_{i}{}^{k}{\bf T}_{kj}.
$$
For the general case, the second author \cite{Li21} calculated the evolution equation of scalar curvature.
\begin{prop} [\cite{Li21}]
Under the flow \eqref{The closed Laplacian flow}, the evolution equation of  scalar curvature $R_{\varphi(t)}$ satisfies
\begin{align}
(\partial_{t}-\triangle_{\varphi(t)})R_{\varphi(t)} = \ & \bigg\{2\left|R_{ij}+\frac{1}{3}
|{\bf T}(t)|^{2}_{\varphi(t)}g_{ij}\right|^{2}_{\varphi(t)}
+\frac{1}{2}\bigg|R_{ijab}R^{ij}{}_{mn}
-\psi_{abmn}\bigg|^{2}_{\varphi(t)}\notag\\
&+\frac{1}{2}
\left|2{\bf T}_{ia}{\bf T}_{jb}R^{ij}{}_{mn}
-\psi_{abmn}\right|^{2}_{\varphi(t)}
+\frac{1}{2}\left|2\widehat{{\bf T}}_{am}\widehat{{\bf T}}_{bn}
-\psi_{abmn}\right|^{2}_{\varphi(t)}\nonumber\notag\\
&+ \ 2|\widehat{{\bf T}}(t)|^{2}_{\varphi(t)}+4|\nabla_{\varphi(t)}
{\bf T}(t)|^{2}_{\varphi(t)}\bigg\}
-\bigg\{|{\rm Rm}_{\varphi(t)}|^{2}_{\varphi(t)}
+\frac{26}{9}R^{2}_{\varphi(t)}\notag\\
&+\frac{1}{2}
\left|R_{ijab}R^{ij}{}_{mn}\right|^{2}_{\varphi(t)}+2\left|{\bf T}_{ia}{\bf T}_{jb}R^{ij}{}_{mn}
\right|^{2}_{\varphi(t)}\notag\\
&+2|\widehat{{\bf T}}(t)|^{4}_{\varphi(t)}+210\bigg\}.\notag
\end{align}
Here $\triangle$ is the Beltrami-Laplace operator, $\psi=\ast_{\varphi}
\varphi$, and $\widehat{\boldsymbol{T}}_{ij}:=\boldsymbol{T}_{ik}
\boldsymbol{T}^{k}{}_{j}$.
\end{prop}
{
\begin{remark}
    The absolute constant $210$ on the right-hand side of the evolution equation is obtained from the contraction of $\psi\ast\psi$ by the metric $g$.
\end{remark}
}
{
The above proposition shows that the evolution equation for $R_{\varphi(t)}$ can be written as 
$$
(\partial_{t}-\triangle_{\varphi(t)})R_{\varphi(t)}=A(t)-B(t)
$$
For some suitable time-dependent, nonnegative functions $A(t)$ and $B(t)$. By the maximum principle, we obtain
\begin{eqnarray*}
    \max_{M}R_{\varphi(0)}+\int_{0}^{t}\max_{M}[A(\tau)-B(\tau)]d\tau&\geq&R_{\varphi(t)}\\
    &\geq&\min_{M}R_{\varphi(0)}+\int_{0}^{t}\min_{M}[A(\tau)-B(\tau)]d\tau.
\end{eqnarray*}}
Observed that the above well-arranged evolution equation can give us a {weak} lower bound for $R_{\varphi(t)}$, which can not prove or disprove the Question \ref{conj4.8} of Lotay and Wei. 

Therefore, {inspired} by the method in \cite{Cao11} and the curvature decomposition
$$
{\rm Rm}=\frac{R}{84}g\circ g+\frac{1}{5}E\circ g+W,
$$
where $W$ is the Weyl tensor, and $E$ is the Einstein tensor or traceless Ricci tensor {which} is defined as
$$
E_{ij}=R_{ij}-\frac{R}{7}g_{ij}.
$$
{We} obtain singularities {that} satisfy the following cases:

\begin{theorem}[\cite{LL23}]\label{thm4.10}
Let $M$ be a {closed} $7$-manifold and $(\varphi(t))_{t\in[0,T)}$ be a solution to the
flow \eqref{The closed Laplacian flow} for closed $G_{2}$-structures. If $T$  is the maximal time and $T< +\infty$, then either
$$
\liminf_{t\to T}\left(\min_{M}R_{\varphi(t)}\right)=-\infty,
$$
or there exists a universal positive constant {$c$} {such that}
$$
R_{\varphi(t)}+{c}>0 \ \ \ \text{and} \ \ \ \limsup_{t\to T}\left(\max_{M}\frac{|W_{\varphi(t)}|_{C^{1}(g(t))}}{R_{\varphi(t)}+{c}}\right)=+\infty.
$$
Here $W_{\varphi(t)}$ is the Weyl tensor of $g_{\varphi(t)}$.
\end{theorem}
\begin{proof}
If there exists a universal positive constant ${c}$ such that $R_{\varphi(t)}+{c}>0$, then by defining
    $$f=\frac{|E_{\varphi(t)}|_{g_{\varphi(t)}}^{2}}{(R_{\varphi(t)}+{c})^{2}}$$
    we can {calculate} the evolution equation of $f$. Using certain technical estimates, we show that the traceless Ricci curvature can be controlled by scalar curvature and the $C^{1}$  norm of the Weyl tensor. More precisely, there exist nonnegative constants $C_{1}$ and $C_{2}$, such that for all $t\in[0,T)$
    \begin{align}\label{4.4)}
        \frac{\displaystyle{\left|{\rm Ric}_{\varphi(t)}-\frac{1}{7}R_{\varphi(t)}\cdot g_{\varphi(t)}\right|_{g_{\varphi(t)}}}}{R_{\varphi(t)}+{c}}\leq C_{1}+C_{2}\max_{M\times[0,t]}\frac{|W_{\varphi(t)}|_{C^{1}(M,g_{\varphi(t)})}}{R_{\varphi(t)}+{c}}.
    \end{align}
    Combining this with the result of Lotay and Wei (Theorem \ref{thm4.5}), we obtain the desired conclusion.
\end{proof}

In fact, Cleyton and Ivanov \cite{CI08} showed that the  Weyl tensor can be decomposed into the following three parts by $G_{2}$-structures:(see also \cite{DGK25} for details)
$$W=W_{77}+W_{64}+W_{27}.$$
Thus, it is worth investigating which part, together with the scalar curvature,  controls the Ricci curvature in Theorem \ref{thm4.10}.

${}$

Building upon the curvature pinching estimate, we establish a partial result to Question \ref{conj4.8}.

\begin{corollary} [\cite{LL23}]
Let $M$ be a closed $7$-manifold and $(\varphi(t))_{t\in[0,T)}$ be a solution to the
flow \eqref{The closed Laplacian flow} for closed $G_{2}$-structures with $T< +\infty$. If the scalar curvature $R_{\varphi(t)}$ and the $C^{1}$  norm of Weyl tensor $W_{\varphi(t)}$ are both uniformly bounded, then the solution $\varphi(t)$ can be extended past time $T$.
\end{corollary}

%%%%%%%%%%%%%

Moreover, we derive a blow-up rate estimate for the 
$C^{1}$ norm of the Weyl tensor.

\begin{theorem} [\cite{LL23}]
Let $M$ be a closed $7$-manifold and $(\varphi(t))_{t\in[0,T)}$ be a solution to the
flow \eqref{The closed Laplacian flow} for closed $G_{2}$-structures. If $T$ is the maximal time and $T< +\infty$, then we have 

    \begin{itemize}
    
        \item[(1)] \text{either}
        $$
        \liminf_{t\rightarrow T}\left(\min_{M} R_{\varphi(t)}\right)=-\infty,
        $$

        \item[(2)] \text{or}
        $$
        \liminf_{t\rightarrow T}\left(\min_{M} R_{\varphi(t)}\right)>-\infty, 
        $$
        but for any positive constants $C>0$ and $\delta>0$, we have
    $$
    \limsup_{t\rightarrow T}\left(\max_{M}|W_{\varphi(t)}|_{C^{1}(M,g_{\varphi(t)})}\right)>\frac{C}{(T-t)^{1-\delta}}.
    $$

    \end{itemize}
\end{theorem}
\begin{proof}
   Suppose there exists a universal positive constant ${c}$ such that 
    $$
    R_{\varphi(t)}+{c}\geq1.
    $$
    Assume that
    $$\limsup_{t\rightarrow T}\left(\max_{M}|W_{\varphi(t)}|_{C^{1}(M,g_{\varphi(t)})}\right)\leq\frac{C}{(T-t)^{1-\delta}}$$
    for some constant $C>0$ and $\delta>0$, Then, by inequality \eqref{4.4)}, it follows that
    \begin{align}
        |{\rm Ric}_{\varphi(t)}|_{g_{\varphi(t)}}\leq\frac{C}{(T-t)^{1-\delta}}\notag.
    \end{align}
    Integrating $g(t)$ in time, there exist{s} some constant $N$, such that
    \begin{align}
        e^{-N}g_{\varphi(0)}\leq g_{\varphi(t)}\leq e^{N}g_{\varphi(0)}\notag
    \end{align}
    for any $t\in [0,T)$. Since the scalar curvature is also uniformly bounded on $[0, T)$, an argument analogous to that in Theorem {\ref{4.4)}} shows that the solution of the Laplacian flow can be extended beyond time $T$, leading to a contradiction.
\end{proof}

%%%%%%%%%%%%%%%%%%%%%%%%%%%%%%%%%%%%%%%%%%%%%%%%gF
\subsection{Gradient estimates and parabolic frequency functionals }
%%%%%%%%%%%%%%%%%%%%%%%%%%%%%%%%%%%%%%%%%%%%%%%%%%%
Gradient estimates provide a versatile tool for studying the analytical, topological, and geometrical properties of manifolds. Therefore, {we} first consider the following Li-Yau type gradient estimate of the heat equation 
\begin{align}
    \label{heat equation}\partial_{t}u(t)=\triangle_{\varphi(t)}u(t)
\end{align}
under the Laplacian flow \eqref{The closed Laplacian flow}, where $\triangle_{\varphi(t)}={\rm tr}_{g_{\varphi(t)}}\left(\nabla^{2}_{g_{\varphi(t)}}\right)$ is the {trace} Laplacian induced by $g_{\varphi(t)}$.

\begin{theorem}[\cite{LLX25}]\label{theorem 1.1}
    Let $(M,\varphi(t))_{t\in(0,T]}$ be the solution of the flow \eqref{The closed Laplacian flow} on a closed $7$-manifold $M$ with $T<+\infty$ and $-Kg_{\varphi(t)}\leq {\rm Ric}_{\varphi(t)}\leq0$, where $g_{\varphi(t)}$ is the Riemannian metric associated with $\varphi(t)$ and $K$ is a positive constant. If $u(t)$ is a positive solution of the heat equation $(\ref{heat equation})$, then on $M\times(0, T]$, the following estimate
    \begin{equation}
        \frac{|\nabla_{g_{\varphi(t)}} u(t)|_{g_{\varphi(t)}}^{2}}{u^{2}(t)}-\alpha\frac{\partial_{t}u(t)}{u(t)}\leq \frac{7\alpha}{2at}+\left(\frac{49\alpha}{3a}+\frac{105\alpha^{2}-98\alpha}{2a(\alpha-1)}+\frac{7\sqrt{29}\alpha}{2\sqrt{ab}}\right)K
    \end{equation}
holds for any $\alpha>1$ {and} $a,b>0$ with $\displaystyle{a+2b=\frac{1}{\alpha}}$.
\end{theorem}
As an application, we derive the Harnack inequality on spacetime and obtain the monotonicity of parabolic frequency functional under the Laplacian flow under bounded Ricci curvature.

Besides, we also consider the parabolic frequency functional for the solution of the linear heat equation 
\begin{equation}
    (\partial_{t}-\triangle_{\varphi(t)})u(t)=a(t)\!\ u(t)\label{lhe}
\end{equation}
under the Laplacian flow with bounded Bakry-\'{E}mery Ricci curvature, where $a(t)$ is a time-dependent smooth function. The parabolic frequency functional is defined by  
\begin{align}
U(t)=\exp\left\{-\int_{t_{0}}^{t}\left[-\frac{2}{3}R_{0}+\frac{h'(s)+\kappa(s)}{h(s)}\right] ds\right\}\frac{\displaystyle{h(t)\int_{M}|\nabla_{g_{\varphi(t)}}u(t)|_{g_{\varphi(t)}}^{2}d \mu_{g_{\varphi(t)}}}}{\displaystyle{\int_{M}u^{2}(t)\!\ d\mu_{g_{\varphi(t)}}}},\notag
\end{align}
where $\displaystyle{R_{0}=\min_{M\times[t_{0},t_{1}]}R_{\varphi(t)}}$, $h(t)$ and $\kappa(t)$ are both time-dependent smooth functions. Then we get the following theorem.

\begin{theorem}[\cite{LLX25}]\label{theorem 1.5}
Let $(M,\varphi(t))_{t\in[0,T]}$ be the solution of the Laplacian flow on a closed $7$-dimensional manifold $M$  with $T<+\infty$ and $\displaystyle{\text{\rm Ric}_{f(t)}\leq \frac{\kappa(t)}{2h(t)}g_{\varphi(t)}}$, where $g_{\varphi(t)}$ is the Riemannian metric associated with $\varphi(t)$. Then the following holds.

\begin{itemize} 

\item[(i)] If $h(t)$ is a negative time-dependent function, then the parabolic frequency $U(t)$ is monotone increasing along the Laplacian flow.

\item[(ii)] If $h(t)$ is a positive time-dependent function, then the parabolic frequency $U(t)$ is monotone decreasing along the Laplacian flow.

\end{itemize}
{
where
$\text{\rm Ric}_{f(t)}=\text{\rm Ric}_{\varphi(t)}+\nabla_{g_{\varphi(t)}}^{2} f(t)$
is the Bakry-\'{E}mery Ricci tensor, $f(t)$ is a time-dependent smooth function on $M$.}
\end{theorem}

The backward uniqueness of solutions to parabolic equations has been a subject of extensive study for over half a century, leading to numerous significant results and applications.
As an application of Theorem \ref{theorem 1.5}, we obtain the following backward uniqueness theorem:
\begin{corollary}[\cite{LLX25}]
Let $(M,\varphi(t))_{t\in[0,T]}$ be the solution of the Laplacian flow on a closed $7$-dimensional manifold $M$  with $T<+\infty$ and
$$
\displaystyle{\text{\rm Ric}_{f(t)}\leq \frac{\kappa(t)}{2h(t)}g_{\varphi(t)}}, 
$$
where $g_{\varphi(t)}$ is the Riemannian metric associated with $\varphi(t)$, $\kappa(t)$ and $h(t)$ are constants dependent only on $t$. Let $u(t)$ be a solution to the heat equation $(\ref{heat equation})$.  If $u(t_{1})=0$, then $u(t)\equiv 0$ for any $t\in[t_{0},t_{1}]\subset(0,T)$.
\end{corollary}
\begin{remark}
    An important aspect in the study of the Laplacian flow is the analysis of Laplacian solitons. As self-similar solutions, they play a key role in understanding singularity formation. We expect that these results on estimates will contribute to further research on Laplacian solitons.
\end{remark}

\subsection{Modified coflow}
%%%%%%%%%%%%%%%%%%%%%%%%%%%%%%%%%%%%%%%%%%%%%%%%%%
In fact, a $G_{2}$-structure and the corresponding metric can also be defined using the 4-form $\psi=\ast_{\varphi}\varphi$. Thus Karigiannis, Mckay, and Tsui \cite{KMT12} introduced the following  Laplacian coflow (with the opposite sign):
\begin{equation}
  \left \{
       \begin{array}{rl}
          \partial_{ t}\psi(t)&=\Delta_{\psi(t)}\psi(t),\\
%           d\varphi(t)&=0,\\
           \psi(0)&=\psi,
       \end{array}
  \right.
  \label{The coflow}
\end{equation}
where $\Delta_{\psi(t)}$ is the Hodge Laplacian induced by $\psi(t)$ and $\psi$ is an initial coclosed $G_{2}$-structure. However, since the Laplacian coflow $\eqref{The coflow}$ is not weakly parabolic and the principal symbol of $\triangle_{\psi}\psi$ is indefinite, it cannot be modified to be parabolic in the same way as the Ricci flow. Therefore, Grigorian \cite{Gri13} introduced the modified Laplacian coflow {of coclosed $G_{2}$-structures} to address this issue:
\begin{equation}
  \left \{
       \begin{array}{rl}
          \partial_{ t}\psi(t)&=\Delta_{\psi(t)}\psi(t)+2d\left[\left(A-{\rm tr} \big{(}\mathbf{T}(t)\big{)}\right)\varphi(t)\right],\\
%           d\varphi(t)&=0,\\
           \psi(0)&=\psi,
       \end{array}
  \right.
  \label{modified Laplacian coflow}
\end{equation}
where ${\rm tr}\big{(}\mathbf{T}(t)\big{)}$ is the trace of the torsion $\mathbf{T}(t)$, $A$ is a {nonnegative} constant, and $\psi$ is an initial coclosed $G_{2}$-structure.
\begin{remark}
    The constant $A$ of the modified Laplacian coflow can be set to zero, adding it may allow for more flexibility. When we set $A$ is a nonnegative constant, the volume of $M$ increases monotonically along the flow.
\end{remark}
{If restricted to coclosed $G_{2}$ structure($d\psi(t)=0$), then the modified Laplacian coflow $\eqref{modified Laplacian coflow}$ preserves the cohomology class of $\psi$. Under the above condition,} Grigorian proved the short-time existence and uniqueness of the modified Laplacian coflow in \cite{Gri13}. In 2018, Chen provided the Shi-type estimate and studied finite-time singularities of this flow in \cite{Chen18}. Bedulli and Vezzoni showed the stability of the modified Laplacian coflow $\eqref{modified Laplacian coflow}$ in \cite{BV20}. As for the real-analyticity, we prove {the following.}
\begin{theorem}[\cite{LL24}]\label{thm4.16}
    {Let $M$ be a compact manifold} and $\psi(t)_{t\in[0,T]}$ be a smooth solution of the modified Laplacian coflow $\eqref{modified Laplacian coflow}$ with coclosed $G_{2}$-structure on an open set $U\subset M$. For each time $t\in(0,T]$, $(U,\psi(t),g_{\psi}(t))$ is real analytic.
\end{theorem}
\begin{proof}
    The core of our proof relies on an improved Shi-type estimate. Let $B_{g_{0}}(p,r)$ be a geodesic ball in $M$ with $p\in M$ and $r>0$. {We can choose a positive constant $\mathbf{M}$, such that $A^{2}\leq \mathbf{M}$ and}
     $$\Lambda(x,t)=\left(|{\rm Rm}|^{2}(x,t)+|\nabla\mathbf{T}|^{2}(x,t)+|\mathbf{T}|^{4}(x,t)\right)^{\frac{1}{2}}\leq \mathbf{M}$$
     on $B_{g_{0}}(p,r)\times[0,T_{\ast}]$ with $0< T_{\ast}\leq\min\{T,1\}$. By some tricks and curvature estimate,
      there exist positive constants $L,C$ depending only on $r,\mathbf{M}$ and $T$, such that
     \begin{align}
         t^{\frac{k}{2}}\left(|\nabla^{k}{\rm Rm}|(x,t)+|\nabla^{k+1}\mathbf{T}|(x,t)\right)\leq CL^{\frac{k}{2}}(k+1)!
     \end{align}
     for all $k\in\mathbb{N}$ and $(x,t)\in B_{g_{0}}(p,r/2)\times[0,T]$.
Finally, by combining the above improved Shi-type estimate with standard methods in analyticity, we establish the real analyticity of solutions to the modified Laplacian coflow.
\end{proof}

We say a $G_{2}$-structure $\psi$ on $M$ is \textit{complete} if the associated metric $g_{\psi}$ is complete. {Analogue to Ricci flow and Laplacian flow,} following Theorem \ref{thm4.16}, we obtain the unique-continuation results for complete solutions.
\begin{corollary}[\cite{LL24}]
    Let $M$ be a connected and simply connected $7$-manifold, and $\psi(t)$, $\tilde{\psi}(t)$ be smooth complete solutions to the modified Laplacian coflow $\eqref{modified Laplacian coflow}$ on $M\times[0,T]$. Then for any $t\in (0,T]$, the followings hold.
    \begin{itemize}
        \item[(i)] If $\psi(t)=\tilde{\psi}(t)$ on some connected open set $U\subset M$, then there exists a diffeomorphism $F$ of $M$ such that $F^{\ast}\tilde{\psi}(t)\equiv\psi(t)$.
        \item[(ii)] Any local diffeomorphism $F:U\rightarrow V$ between connected open sets $U,V\subset M$ satisfying $F^{\ast}(\psi(t)|_{V})=\psi(t)|_{U}$ can be uniquely extended to a global diffeomorphism $F$ of $M$ with $F^{\ast}(\psi(t))=\psi(t)$.
    \end{itemize}
\end{corollary}

\textbf{Acknowledgments.}\ \ 
The second author is funded by the Shanghai Institute for Mathematics and Interdisciplinary Sciences (SIMIS) under grant number SIMIS-ID-2025-AD. The authors would also like to thank the referee for the valuable comments and suggestions.

%%%%%%%%%%%%%%%%%%%%%%%%%%%%%%%%%%%%%%%%%%%%%%%%%%%%%%%%%%%%%%%%%%%%%%%%%%%%%%
%%%%%%%%%%%%%%%%%%%%%%%%%%%%%%%%%%%%%%%%%%%%%%%%%%%%%%%%%%%%%%%%%%%%%%%%%%%%%%

\end{document}